\documentclass[12pt]{amsart}

\usepackage{amsmath, amsthm, amssymb, tikz}
\usepackage[letterpaper,left=2.25cm,right=2.25cm,top=3cm,bottom=3cm,headsep=1cm]{geometry}
\usetikzlibrary{topaths, calc, shapes,arrows,fit,positioning}
\usepackage{ytableau}
\usepackage{graphicx}
\usepackage{xparse}
\usepackage{mathtools}
\usepackage{hyperref}
\usepackage{subcaption}

\tikzset{Text/.style={black}}
\tikzset{Vert/.style={circle,fill=orange!50, scale=0.7, draw=black}}

\hypersetup{colorlinks, linkcolor={red!75!black},
citecolor={blue!50!black}, urlcolor={blue!80!black}}

\newcommand{\Z}{\mathbb{Z}}
\newcommand{\M}{\mathcal{M}}
\newcommand{\G}{\mathcal{G}}
\newcommand{\inv}{^{-1}}
\newcommand{\lcm}{\textrm{lcm}}
\newcommand{\init}{\textrm{init}}
\newcommand{\reg}{\textrm{reg}}

    \newtheorem{theorem}{Theorem}[section]
    \newtheorem{lemma}[theorem]{Lemma}
    \newtheorem{prop}[theorem]{Proposition}
    
    \newtheorem*{thm*}{Theorem}
\theoremstyle{definition}
    \newtheorem{definition}[theorem]{Definition}
    \newtheorem{example}[theorem]{Example}
\theoremstyle{remark}
    \newtheorem{remark}[theorem]{Remark}
\numberwithin{equation}{section}

\begin{document}

\title{Matrix Schubert varieties, binomial ideals, \\ and reduced Gr\"obner bases}
\date{June 2, 2023}
\author{Ada Stelzer}
\address{Department of Mathematics, University of Illinois at Urbana-Champaign, 1409 W.~Green Street, Urbana, IL 61801}
\email{astelzer@illinois.edu}
\maketitle

\pagestyle{plain}
\begin{abstract}
    We prove a sharp lower bound on the number of terms in an element of the reduced Gr\"obner basis of a Schubert determinantal ideal $I_w$ under the term order of [Knutson--Miller `05]. We give three applications. 
    First, we give a pattern-avoidance characterization of the matrix Schubert varieties whose defining ideals are binomial. This complements a result of [Escobar--M\'esz\'aros '16] on matrix Schubert varieties that are toric with respect to their natural torus action. Second, we give a combinatorial proof that the recent formulas of [Rajchgot--Robichaux--Weigandt `23] and [Almousa--Dochtermann--Smith `22] computing the Castelnuovo-Mumford regularity of vexillary $I_w$ and toric edge ideals of bipartite graphs respectively agree for binomial $I_w$.
    Third, we demonstrate that the Gr\"obner basis for $I_w$ given by minimal generators [Gao--Yong `22] is reduced if and only if the defining permutation $w$ is vexillary. 
\end{abstract}

\section{Introduction and main results}\label{section1}
Let $\Bbbk$ be a field and let $\M_n$ be the affine space of $n\times n$ matrices with entries in $\Bbbk$. For a permutation $w$ in the symmetric group $S_n$, the  \textit{permutation matrix} $M_w$ is the element of $\M_n$ with $1$s in the positions $(i, w(i))$ for $1\leq i\leq n$ and $0$s elsewhere. Let $B_+$ and $B_-$ denote the groups of invertible upper- and lower-triangular $n\times n$ matrices over~$\Bbbk$. The product $B_-\times B_+$ acts on $\M_n$ by left- and right-multiplication, \textit{i.e.}, $(x, y)\cdot A = x A y\inv$. Fulton \cite{Fulton} introduced the following object in his study of degeneracy loci of flagged vector bundles.

\begin{definition}
    The \textit{matrix Schubert variety} $X_w$ is the Zariski closure of the $B_-\times B_+$-orbit of the permutation matrix $M_w$ in $\M_n\simeq \Bbbk^{n^2}$.
\end{definition}

\begin{definition}
    The \textit{Schubert determinantal ideal} $I_w$ is the ideal of $R = \Bbbk[x_{ij}]_{1\leq i, j\leq n}$ corresponding to $X_w$.
\end{definition}

For a permutation $w$, let $r_w(i, j) = r(i, j) = r_{ij}$ be the \textit{rank function} counting the number of 1s weakly northwest of position $(i, j)$ in the permutation matrix $M_w$. Let $M^{[a, b]}$ denote the northwest $a\times b$ submatrix of the generic matrix of variables $[x_{ij}]_{1\leq i, j\leq n}$. In \cite{Fulton}, Fulton described a generating set for $I_w$:
\begin{equation*}
    I_w = \langle (r_{ij}+1)\times(r_{ij}+1)\textrm{ minors of } M^{[i, j]}| 1\leq i, j\leq n\rangle.
\end{equation*}

\begin{example}\label{ex:intro}
    Consider $I_w$ for $w = 31425$. The rank function is:
    
    \begin{center}
        $r_w=$ \begin{tabular}{|c|c|c|c|c|}
        \hline
        0 & 0 & 1 & 1 & 1\\
        \hline
        1 & 1 & 2 & 2 & 2\\
        \hline
        1 & 1 & 2 & 3 & 3\\
        \hline
        1 & 2 & 3 & 4 & 4\\
        \hline
        1 & 2 & 3 & 4 & 5\\
        \hline
        \end{tabular}
    \end{center}
   
    $I_w$ is generated by the two $1\times 1$ minors of $M^{[1, 2]}$ and the three $2\times 2$ minors of $M^{[3, 2]}$:
    $$I_w = \bigg\langle x_{11}, x_{12},
        \begin{vmatrix}
            x_{11} & x_{12}\\
            x_{21} & x_{22}\\
        \end{vmatrix},
        \begin{vmatrix}
            x_{11} & x_{12}\\
            x_{31} & x_{32}\\
        \end{vmatrix},
        \begin{vmatrix}
            x_{21} & x_{22}\\
            x_{31} & x_{32}\\
        \end{vmatrix}\bigg\rangle.$$
\end{example}

In \cite[Theorem B]{KnutsonMiller}, Knutson--Miller showed that Fulton's generators form a \textit{Gr\"obner basis} for $I_w$ under ``antidiagonal" term orderings. More recently, Gao--Yong refined Fulton's generators to a minimal generating set that is still a Gr\"obner basis in \cite[Corollary 1.8]{GaoYong}. Now, the \textit{reduced} Gr\"obner basis of an ideal is unique (for a given term order) and contains fewer terms than any other Gr\"obner basis. Our first theorem establishes a sharp lower bound for the number of terms in the reduced Gr\"obner basis of $I_w$ under any antidiagonal term order.

\begin{theorem}\label{thm:main}
        The reduced Gr\"obner basis $\G'_w$ of $I_w$ under any antidiagonal term order has one generator for each element of the Gao--Yong Gr\"obner basis $\G_w$. Each generator of degree $d$ in $\G'_w$ has at least $2^{d-1}$ terms.
\end{theorem}

By minimality, any Gr\"obner basis using antidiagonal term order must have as many generators as the Gao--Yong Gr\"obner basis, and Theorem~\ref{thm:main} says the number of terms in each generator must be exponential in its degree. This suggests that any description of $\G'_w$ (under antidiagonal term order) should be complicated.

We present three applications of Theorem~\ref{thm:main}.

The first two applications are stated in terms of permutation pattern avoidance. A permutation $w\in S_n$ \textit{contains} a pattern $v\in S_m$ if there exist indices $i_1<i_2<\dots<i_m$ such that $w(i_1),\dots,w(i_m)$ appear in the same relative order as $v(1),\dots,v(m)$. If $w$ does not contain the pattern $v$, then we say it \textit{avoids} $v$. Permutations avoiding $2143$ are called \textit{vexillary}. 

\begin{theorem}\label{thm:vexillary}
    The Gao--Yong Gr\"obner basis $\G_w$ for $I_w$ is reduced if and only if $w$ is vexillary.
\end{theorem}

The number of vexillary permutations in $S_n$ is asymptotic to $c9^nn^{-4}$ for some constant $c$ as shown by Macdonald \cite[pg. 22]{MacDonaldNotes}, while the number of permutations in $S_n$ is asymptotic to $\sqrt{2\pi n}\left(\frac{n}{e}\right)^n$ by Stirling's approximation. Theorem~\ref{thm:vexillary} therefore shows that the Gao--Yong Gr\"obner basis is reduced for a super-exponentially small percentage of $w\in S_n$ as $n\to\infty$.

All matrix Schubert varieties are normal \cite{Fulton}.
Therefore, those that are (affine) toric varieties with respect to some algebraic torus correspond exactly to \emph{binomial ideals} $I_w$ 
(meaning they can be generated by binomials).
Our next theorem
characterizes when $I_w$ is binomial:

\begin{theorem}\label{thm:binomial}
    The Schubert determinantal ideal $I_w$ is binomial if and only if $w$ avoids the patterns $1243$ and $2143$.
\end{theorem}

In \cite[Theorem 3.4]{ToricMatrixSchubs}, Escobar--M\'esz\'aros presented a combinatorial characterization of matrix Schubert varieties that are toric with respect to a certain natural torus action. This characterization leaves open the possibility of other matrix Schubert varieties that are toric with respect to a different action.  Theorem~\ref{thm:binomial} complements their result by showing that the toric matrix Schubert varieties Escobar--M\'esz\'aros identified are the only ones that exist.

By Theorem~\ref{thm:binomial}, toric matrix Schubert varieties are a subclass of vexillary matrix Schubert varieties. Special tools apply in the vexillary case (see \cite{KMY} and the references therein). 

The permutations in $S_n$ avoiding the patterns $1243$ and $2143$ have been previously studied. In \cite[Corollary 9]{Kremer}, Kremer proved that they are enumerated by the \textit{large Schr\"oder numbers} $s_{n-1}$ (OEIS sequence A006318). The large Schr\"oder numbers have generating function
\begin{equation*}
    G(x) = \frac{1-x-\sqrt{x^2-6x+1}}{2x}
\end{equation*}
and satisfy the recurrence relation

\begin{equation*}
    s_n = \frac{6n-3}{n+1}s_{n-1}-\frac{n-2}{n+1}s_{n-2}\quad(n\geq 2).
\end{equation*}

As a third application we consider formulas for the \textit{(Castelnuovo--Mumford) regularity} of binomial $I_w$.  Regularity is a homological invariant that roughly describes the complexity of a module. In \cite[Theorem 1.5]{SchubReg}, Rajchgot--Robichaux--Weigandt gave a formula for the regularity of vexillary $I_w$\footnote{Pechenik--Speyer--Weigandt gave a formula for the regularity of \textit{all} $I_w$ in \cite[Theorem 1.1]{SchubReg2}.}, while in \cite[Corollary 6.7]{ToricReg} Almousa--Dochtermann--Smith gave a formula for the regularity of toric edge ideals of bipartite graphs. By Theorem~\ref{thm:binomial} and Portakal's interpretation of binomial $I_w$ as toric edge ideals \cite[pg. 7]{ToricEdgeIdeals}, these two formulas must agree for binomial $I_w$. We give a direct proof of this fact.

\subsection*{Organization}
Section~\ref{section2} contains preliminary definitions and results needed for the proof of Theorem~\ref{thm:main}. We introduce standard facts about reduced Gr\"obner bases from \cite{CLO} and recall work of Gao--Yong \cite{GaoYong} that produces a Gr\"obner basis for $I_w$ that is also a minimal generating set. Section~
\ref{section3} contains the proof of Theorem~\ref{thm:main}. In Section~\ref{section4} we prove Theorem~\ref{thm:vexillary}, and Section~\ref{section5} contains the proof of Theorem~\ref{thm:binomial}. Section~\ref{section6} reviews regularity along with the formulas of Rajchgot--Robichaux--Weigandt \cite{SchubReg} and Almousa--Dochtermann--Smith~\cite{ToricReg} before proving their formulas agree for binomial $I_w$.

\section{Background}\label{section2}
\subsection{Reduced Gr\"obner bases}
    We review standard facts about Gr\"obner bases needed in our proofs. We use Cox, Little, and O'Shea's book \cite{CLO} as our reference, following their terminology and notation. All ideals in this section belong to the ring $\Bbbk[x_1,\dots,x_n]$.

    \begin{definition}
        The \textit{lead term} of a polynomial $f$ with respect to a term order $<$ on $\Bbbk[x_1,\dots,x_n]$ is denoted $LT(f)$. The \textit{initial ideal} of an ideal $I$ is $\init(I) = \langle LT(f)|f\in I\rangle$.
    \end{definition}
    
    \begin{definition}
        A \textit{Gr\"obner basis} for an ideal $I$ with respect to a term order $<$ is a finite subset $\G = \{g_1,\dots,g_s\}\subseteq I$ such that $\init(I)$ is generated by $\{LT(g_1),\dots,LT(g_s)\}$.
    \end{definition}
    
    \begin{definition}
        Given two polynomials $f, g\in \Bbbk[x_1,\dots,x_n]$, their \textit{S-polynomial} is
        \begin{equation*}
            S(f, g) = \frac{\lcm(LT(f), LT(g))}{LT(f)}f - \frac{\lcm(LT(f), LT(g))}{LT(g)}g.
        \end{equation*}
    \end{definition}

    The results in the next theorem appear as Corollary 2.5.6 and Theorem 2.7.2 in \cite{CLO}. We write $\overline{f}^A$ for the remainder when $f$ is divided by the (ordered) elements of $A$.
    
    \begin{theorem}\label{thm:grobref1}
        Let $I\subseteq \Bbbk[x_1,\dots,x_n]$ be any ideal. Then
        \begin{enumerate}
            \item $I$ has a Gr\"obner basis.
            \item Every Gr\"obner basis for $I$ is a generating set for $I$.
            \item (Buchberger's Algorithm) Any generating set $G$ for $I$ can be enlarged to a Gr\"obner basis $\G$ via the following finite algorithm. Begin with $\G_0 = G$. Iteratively set $\G_i = \G_{i-1}\cup\{\overline{S(g, h)}^{\G_{i-1}}\}$, where $g$ and $h$ are elements of $\G_{i-1}$ such that $\overline{S(g, h)}^{\G_{i-1}}\neq 0$. Return $\G = \G_k$ when there are no such elements $g$ and $h$ in $\G_k$. 
        \end{enumerate}
    \end{theorem}
    
    Gr\"obner bases are non-minimal and non-unique in general, but with additional steps we can identify a special type of Gr\"obner basis which is unique for a given ideal and term order.
    
    \begin{definition}
        A Gr\"obner basis $\G$ for an ideal $I$ with respect to a term order $<$ is \textit{minimal} if $LT(g_1)$ does not divide $LT(g_2)$ for all $g_1,g_2\in\G$. Equivalently, $\G$ is a minimal Gr\"obner basis if the lead terms of its elements form the unique minimal generating set for $\init(I)$.
    \end{definition}

    \begin{remark}
        If $\G$ is a non-minimal Gr\"obner basis for $I$, then for some $g\in \G$, $\G\setminus\{g\}$ is a Gr\"obner basis for $I$. Since all Gr\"obner bases for $I$ are generating sets by Theorem~\ref{thm:grobref1}(2), it follows that $\G$ is not a minimal generating set for $I$. Thus any minimal generating set for $I$ that is a Gr\"obner basis is a minimal Gr\"obner basis. The converse does not hold in general.
    \end{remark}

    \begin{definition}
        A Gr\"obner basis $\G'$ for an ideal $I$ with respect to a term order $<$ is \textit{reduced} if for any two generators $g_1, g_2\in\G'$, $LT(g_1)$ does not divide any term of $g_2$.
    \end{definition}
    
    \begin{theorem}[{\cite[Theorem 5, Section 2.7]{CLO}}]\label{thm:grobref2} Let $I$ be an ideal. Then
        \begin{enumerate}
            \item $I$ has a unique reduced Gr\"obner basis $\G'$ up to scalar multiplication.
            \item Any Gr\"obner basis $\G$ for $I$ can be reduced to $\G'$ via the following finite algorithm. First reduce $\G$ to a minimal Gr\"obner basis $\G_0$ by removing each $g\in \G$ such that $LT(h)$ divides $LT(g)$ for some $h\in \G$. Say $\G_0 = \{g_1,\dots,g_k\}$. For $1\leq i\leq k$ let $\G'_i = (\G_{i-1}\setminus\{g_i\})$, let $g'_i = \overline{g_i}^{\G'_{i-1}}$, and set $\G_i = \G'_i\cup\{g'_i\}$. Then $\G_k = \G'$ is the reduced Gr\"obner basis of $I$.
        \end{enumerate}
    \end{theorem}
    
    For the proof of Theorem~\ref{thm:binomial} we need an additional fact characterizing the reduced Gr\"obner bases of binomial ideals. This appears as Proposition 1.1(a) in \cite{EisenbudSturmfels}. However, since the proof is short we provide it for convenience.

    \begin{prop}[{\cite[Proposition~1.1(a)]{EisenbudSturmfels}}]\label{prop:binomial}
        If $I$ is a binomial ideal, then for any term order $<$ the reduced Gr\"obner basis $\G'$ of $I$ with respect to $<$ consists of binomials.
    \end{prop}
    \begin{proof}
        If $I$ is a binomial ideal, then $I = (m_1-m'_1,\dots,m_k-m'_k)$ for some monomials $m_i$ and $m'_i$ with $m_i > m'_i$ for each $i$. Buchberger's algorithm extends this generating set to a Gr\"obner basis by iteratively adding new generators of the form
        \begin{equation}\nonumber
            \frac{\lcm(m_i, m_j)}{m_i}(m_i-m'_i)-\frac{\lcm(m_i, m_j)}{m_j}(m_j-m'_j) = \frac{\lcm(m_i, m_j)}{m_j}m'_j - \frac{\lcm(m_i, m_j)}{m_i}m'_i.
        \end{equation}
        
        Since each new generator is itself a binomial, Buchberger's algorithm constructs a Gr\"obner basis $\G$ consisting only of binomials. The algorithm in Theorem~\ref{thm:grobref2}(2) then reduces $\G$ to the reduced Gr\"obner basis $\G'$ by removing generators and performing polynomial divisions. This can never increase the number of terms in a generator. Thus $\G'$ consists of binomials.
    \end{proof}
    
\subsection{Generators for $I_w$}
    Our main results build on prior work concerning Gr\"obner bases of Schubert determinantal ideals $I_w$. We collect the key facts needed for our proof below.

    \begin{definition}
        The \textit{Rothe diagram} of a permutation $w\in S_n$ is the set $D(w) = \{(i, j)\in [n]\times[n]: j<w(i), i<w\inv(j)\}$.
    \end{definition}

    \begin{definition}[\cite{Fulton}]
        For a permutation $w\in S_n$, \textit{Fulton's essential set} is the subset $E(w)\subseteq D(w)$ consisting of pairs $(i, j)\in D(w)$ such that $(i+1, j)$ and $(i, j+1)$ are not in $D(w)$. Visually, $E(w)$ consists of all ``southeast corners" of connected components of $D(w)$. 
    \end{definition}

    The first part of Theorem~\ref{thm:schubref} below was proved by Fulton as \cite[Lemma 3.10]{Fulton}, while the second part was established by Knutson and Miller as \cite[Theorem B]{KnutsonMiller}.\\
    
    \begin{theorem}[\cite{Fulton, KnutsonMiller}]\label{thm:schubref}
        Let $w\in S_n$ be a permutation. Then
        \begin{enumerate}
            \item $I_w$ is generated by the $(r_{ij}+1)\times(r_{ij}+1)$ minors of $M^{[i, j]}$ as $(i, j)$ varies over the essential set $E(w)$.
            \item This generating set forms a Gr\"obner basis for $I_w$ under any antidiagonal term order (i.e., any term order such that the antidiagonal of a generic minor is the lead term).
        \end{enumerate}
    \end{theorem}
    
    \begin{example}
        Consider the permutations $w = 14235$ and $v = 31254$. Their Rothe diagrams are pictured below, with the values of the rank function $r_{ij}$ displayed only in the elements of the essential sets.
    
    
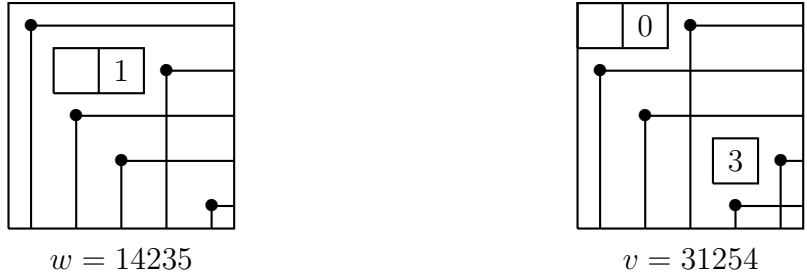
\begin{figure}[h]
    \centering
    \begin{subfigure}{0.45\textwidth}
    \centering
    \begin{tikzpicture}[scale = 0.30]
       \draw[black, thick] (0,0) -- (10,0) -- (10,10) -- (0,10) -- (0,0);
        \draw[black, thick] (2,6) -- (6,6) -- (6,8) -- (2,8) -- (2,6);
        \draw[black, thick] (4, 6) -- (4, 8);
        \draw[black, thick] (1, 9) -- (1,0);
        \draw[black, thick] (1, 9) -- (10,9);
        \draw[black, thick] (3, 5) -- (3,0);
        \draw[black, thick] (3, 5) -- (10,5);
        \draw[black, thick] (5, 3) -- (5,0);
        \draw[black, thick] (5, 3) -- (10,3);
        \draw[black, thick] (7, 7) -- (7,0);
        \draw[black, thick] (7, 7) -- (10,7);
        \draw[black, thick] (9, 1) -- (9,0);
        \draw[black, thick] (9, 1) -- (10,1);
        \draw (1,9) node{$\bullet$};
        \draw (3,5) node{$\bullet$};
        \draw (5,3) node{$\bullet$};
        \draw (7,7) node{$\bullet$};
        \draw (9, 1) node{$\bullet$};
        \draw (5, 7) node{1};
    \end{tikzpicture}
    \caption*{\normalsize{$w = 14235$}}
    \end{subfigure}
    \begin{subfigure}{0.45\textwidth}
    \centering
    \begin{tikzpicture}[scale = 0.30]
        \draw[black, thick] (0,0) -- (10,0) -- (10,10) -- (0,10) -- (0,0);
        \draw[black, thick] (0, 8) -- (4, 8) -- (4, 10) -- (0, 10) -- (0, 8);
        \draw[black, thick] (6, 2) -- (8, 2) -- (8, 4) -- (6, 4) -- (6, 2);
        \draw[black, thick] (2, 8) -- (2, 10);
        \draw[black, thick] (1, 7) -- (1,0);
        \draw[black, thick] (1, 7) -- (10,7);
        \draw[black, thick] (3, 5) -- (3,0);
        \draw[black, thick] (3, 5) -- (10,5);
        \draw[black, thick] (5, 9) -- (5,0);
        \draw[black, thick] (5, 9) -- (10,9);
        \draw[black, thick] (7, 1) -- (7,0);
        \draw[black, thick] (7, 1) -- (10,1);
        \draw[black, thick] (9, 3) -- (9,0);
        \draw[black, thick] (9, 3) -- (10,3);
        \draw (1,7) node{$\bullet$};
        \draw (3,5) node{$\bullet$};
        \draw (5,9) node{$\bullet$};
        \draw (7,1) node{$\bullet$};
        \draw (9, 3) node{$\bullet$};
        \draw (3, 9) node{0};
        \draw (7, 3) node{3};
    \end{tikzpicture}
    \caption*{\normalsize{$v = 31254$}}
    \end{subfigure}
    \vspace{-3mm}
    \end{figure}
    
    Theorem~\ref{thm:schubref} asserts that $I_w$ is generated by the three $2\times2$ minors of $M^{[2, 3]}$. Similarly, $I_v$ is generated by the two $1\times1$ minors of $M^{[1, 2]}$ and the one $4\times4$ minor of $M^{[4, 4]}$.
    \end{example}

    In \cite[Theorem 1.6]{GaoYong}, Gao and Yong refine Fulton's generators to a \emph{minimal generating set}, that is, one where no generator can be removed without changing the ideal. To state their result, which is presented as Theorem~\ref{thm:schubertgrobner} below, we need some additional definitions. For sets $I, J\subseteq [n]$, let $m_{I,J}$ denote the minor of the generic matrix $[x_{ij}]_{1\leq i, j\leq n}$ using rows $I$ and columns $J$. We say $m_{I, J}$ \textit{belongs} to some $(i, j)$ in the essential set $E(w)$ if it is a Fulton generator associated to $(i, j)$, i.e. if $I\subseteq [i]$, $J\subseteq [j]$, and the rank of $m_{I, J}$ is $r(i, j)+1$.
    
    \begin{definition}
        A minor $m_{I, J}$ \textit{attends} $M^{[i', j']}$ if $|I\cap [i']| > r(i', j')$ and $|J\cap [j']| = r(i, j)+1$, or if $|I\cap[i']| = r(i, j)+1$ and $|J\cap[j']| > r(i', j')$.
    \end{definition}
    
    \begin{definition}
        A minor $m_{I, J}$ belonging to $(i, j)\in E(w)$ is \textit{elusive} if it does not attend $M^{[i', j']}$ for all elements $(i', j')$ in $E(w)$ such that $r(i', j') < r(i, j)$. 
    \end{definition}
    
    \begin{theorem}[{\cite[Theorem 1.6 and Corollary 1.8]{GaoYong}}]\label{thm:schubertgrobner}
        The Schubert determinantal ideal $I_w$ is minimally generated by the set $\G_w$ of elusive minors, which includes at least one minor with southeast corner $b$ for each $b\in D(w)$. $\G_w$ forms a Gr\"obner basis with respect to any antidiagonal term order.
    \end{theorem}

\section{The reduced Gr\"obner basis for $I_w$}\label{section3}
    The work of Gao--Yong \cite{GaoYong} summarized in Theorem~\ref{thm:schubertgrobner} describes a minimal Gr\"obner basis for any Schubert determinantal ideal $I_w$. The following examples illustrate that although the minimal Gr\"obner basis $\G_w$ of elusive minors and the reduced Gr\"obner basis $\G'_w$ have the same number of generators, they may not be identical. 
    
    \begin{example}
        Consider $I_w$ for $w = 31542$, which has the following Rothe diagram:
        
        \begin{figure}[h]
        \centering
        \begin{tikzpicture}[scale = 0.30]
           \draw[black, thick] (0,0) -- (10,0) -- (10,10) -- (0,10) -- (0,0);
            \draw[black, thick] (0,8) -- (4,8) -- (4,10) -- (0,10) -- (0,8);
            \draw[black, thick] (2, 2) -- (4, 2) -- (4, 6) -- (2, 6) -- (2, 2);
            \draw[black, thick] (6, 4) -- (8, 4) -- (8, 6) -- (6, 6) -- (6, 4);
            \draw[black, thick] (2, 8) -- (2, 10);
            \draw[black, thick] (2, 4) -- (4, 4);
            \draw[black, thick] (1, 7) -- (1,0);
            \draw[black, thick] (1, 7) -- (10,7);
            \draw[black, thick] (3, 1) -- (3,0);
            \draw[black, thick] (3, 1) -- (10,1);
            \draw[black, thick] (5, 9) -- (5,0);
            \draw[black, thick] (5, 9) -- (10,9);
            \draw[black, thick] (7, 3) -- (7,0);
            \draw[black, thick] (7, 3) -- (10,3);
            \draw[black, thick] (9, 5) -- (9,0);
            \draw[black, thick] (9, 5) -- (10,5);
            \draw (1,7) node{$\bullet$};
            \draw (3,1) node{$\bullet$};
            \draw (5,9) node{$\bullet$};
            \draw (7,3) node{$\bullet$};
            \draw (9,5) node{$\bullet$};
            \draw (3, 9) node{0};
            \draw (3, 3) node{1};
            \draw (7, 5) node{2};
        \end{tikzpicture}
        \end{figure}
        
        The Knutson--Miller Gr\"obner basis $\G_1$ consists of two $1\times 1$ minors, ${4\choose 2}$ $2\times 2$ minors, and ${4\choose 3}$ $3\times 3$ minors. The minimal Gr\"obner basis $\G_w$ of Gao--Yong refines this to the following:
        $$\G_w = \left\{x_{11}, x_{12}, 
        \begin{vmatrix}
            x_{21} & x_{22}\\
            x_{31} & x_{32}\\
        \end{vmatrix},
        \begin{vmatrix}
            x_{21} & x_{22}\\
            x_{41} & x_{42}\\
        \end{vmatrix},
        \begin{vmatrix}
            x_{31} & x_{32}\\
            x_{41} & x_{42}\\
        \end{vmatrix},
        \begin{vmatrix}
            x_{11} & x_{13} & x_{14}\\
            x_{21} & x_{23} & x_{24}\\
            x_{31} & x_{33} & x_{34}
        \end{vmatrix},
        \begin{vmatrix}
            x_{12} & x_{13} & x_{14}\\
            x_{22} & x_{23} & x_{24}\\
            x_{32} & x_{33} & x_{34}
        \end{vmatrix}\right\}.$$
        
        Note that some terms of the $3\times 3$ minors are divisible by the generators $x_{11}$ and $x_{12}$, so this Gr\"obner basis is not reduced. Applying the algorithm of Theorem~\ref{thm:grobref2}(2), we find that in the reduced Gr\"obner basis $\G'_w$ the degree-$3$ generators each have only four terms. Explicitly,
        \begin{multline*}
            \G'_w = \biggl\{ x_{11}, x_{12}, 
            \begin{vmatrix}
                x_{21} & x_{22}\\
                x_{31} & x_{32}\\
            \end{vmatrix},
            \begin{vmatrix}
                x_{21} & x_{22}\\
                x_{41} & x_{42}\\
            \end{vmatrix},
            \begin{vmatrix}
                x_{31} & x_{32}\\
                x_{41} & x_{42}\\
            \end{vmatrix},\\
            x_{14}\begin{vmatrix}
                x_{21} & x_{23}\\
                x_{31} & x_{33}\\
            \end{vmatrix}
            - x_{13}\begin{vmatrix}
                x_{21} & x_{24}\\
                x_{31} & x_{34}\\
            \end{vmatrix},\ 
            x_{14}\begin{vmatrix}
                x_{22} & x_{23}\\
                x_{32} & x_{33}\\
            \end{vmatrix}
            - x_{13}\begin{vmatrix}
                x_{22} & x_{24}\\
                x_{32} & x_{34}\\
            \end{vmatrix}\biggr\}.
        \end{multline*}
    \end{example}

    \begin{example}\label{ex:extreme}
        Consider $I_w$ for $w = 32154$. The Rothe diagram for $w$ is pictured below.
        
        \begin{figure}[h]
        \centering
        \begin{tikzpicture}[scale = 0.30]
            \draw[black, thick] (0,0) -- (10,0) -- (10,10) -- (0,10) -- (0,0);
            \draw[black, thick] (0, 6) -- (2, 6) -- (2, 8) -- (0, 8) -- (0, 6);
            \draw[black, thick] (0, 8) -- (2, 8) -- (2, 10) -- (0, 10) -- (0, 8);
            \draw[black, thick] (2, 8) -- (4, 8) -- (4, 10) -- (2, 10) -- (2, 8);
            \draw[black, thick] (6, 2) -- (8, 2) -- (8, 4) -- (6, 4) -- (6, 2);
            \draw[black, thick] (1, 5) -- (1,0);
            \draw[black, thick] (1, 5) -- (10,5);
            \draw[black, thick] (3, 7) -- (3,0);
            \draw[black, thick] (3, 7) -- (10,7);
            \draw[black, thick] (5, 9) -- (5,0);
            \draw[black, thick] (5, 9) -- (10,9);
            \draw[black, thick] (7, 1) -- (7,0);
            \draw[black, thick] (7, 1) -- (10,1);
            \draw[black, thick] (9, 3) -- (9,0);
            \draw[black, thick] (9, 3) -- (10,3);
            \draw (1,5) node{$\bullet$};
            \draw (3,7) node{$\bullet$};
            \draw (5,9) node{$\bullet$};
            \draw (7,1) node{$\bullet$};
            \draw (9,3) node{$\bullet$};
            \draw (1, 7) node{0};
            \draw (3, 9) node{0};
            \draw (7, 3) node{3};
        \end{tikzpicture}
        
        \end{figure}
        
        The Knutson--Miller Gr\"obner basis is equal to $\G_w$ in this case, consisting of the $4\times 4$ minor $g = \det M^{[4, 4]}$ along with the three $1\times 1$ minors $x_{11}$, $x_{12}$, $x_{21}$. The reduced Gr\"obner basis $\G'_w$ is obtained by removing all terms of $g$ containing these three variables, leaving behind a degree-$4$ generator with $8=2^{4-1}$ terms.
    \end{example}

    Proposition~\ref{prop:extreme} below generalizes Example~\ref{ex:extreme} (where $n=4$), proving that the lower bound on the number of terms in elements of $\G'_w$ given by Theorem~\ref{thm:main} is sharp. Our proof of Theorem~\ref{thm:main} relies on reduction to this special case.

    \begin{prop}\label{prop:extreme}
        Let $w = (n-2)(n-3)\dots(2)(1)(n+1)(n)$ for some $n\geq 3$. Then the Gao--Yong Gr\"obner basis for $I_w$ is $\G_w = \{x_{ij}\}_{i+j\leq n-1}\cup\{g\}$, where $g = \det{M^{[n, n]}}$. The reduced Gr\"obner basis is $\G'_w = \{x_{ij}\}_{i+j\leq n-1}\cup\{g'\}$, where $g'$ has degree $n$ and $2^{n-1}$ terms.
    \end{prop}
    \begin{proof}
        As in Example~\ref{ex:extreme}, it follows immediately from definitions that the Gao--Yong Gr\"obner basis $\G_w$ has the claimed form in this case. Dividing $g$ by $\G_w\setminus\{g\}$ leaves a remainder $g'$ consisting of all terms in $g$ not containing any $x_{ij}$ with $i+j\leq n-1$. By the reduction algorithm of Theorem~\ref{thm:grobref2}(2) it follows that $\G'_w = (\G_w\setminus\{g\})\cup\{g'\}$. It remains only to show that $g'$ contains $2^{n-1}$ terms. Since $g$ is the determinant of an $n\times n$ matrix, terms of $g$ are in bijection with permutations in $S_n$. Permutations $v\in S_n$ corresponding to terms avoiding $\{x_{ij}\}_{i+j\leq n-1}$ are those satisfying $v(i)\geq n-i$ for all $i$. Constructing these permutations by iteratively choosing the value of $v(i)$ with two options at each step except the last shows that there are precisely $2^{n-1}$ terms in $g'$ as claimed.
    \end{proof}

    In order to reduce the proof of Theorem~\ref{thm:main} to Proposition~\ref{prop:extreme}, we employ three lemmas which establish properties of elusive minors and their positioning relative to each other. For all of these lemmas, let $m_{I, J}$ be an elusive minor of rank $d$ in $\G_w$ with row indices $I = \{i_1,\dots,i_d\}$ and column indices $J = \{j_1,\dots,j_d\}$.

    \begin{lemma}\label{lemma:elusive1}
        Let $m_{I, J}$ be as above. If $(i_a, j)$ lies in the Rothe diagram $D(w)$ for some $1\leq a < d$ and $j\geq j_d$, then $r_w(i_a, j)\geq a$. Similarly, if $(i, j_b)$ lies in $D(w)$ for some $1\leq b < d$ and $i\geq i_d$, then $r_w(i, j_b)\geq b$.
    \end{lemma}
    \begin{proof}
        This is immediate from the definition of an elusive minor. If $r(i_a, j)<a$ for some $(i_a, j)\in D(w)$ with $1\leq a\leq d$ and $j\geq j_d$, then $m_{I, J}$ attends $M^{[i_a, j]}$. Any minor that attends an element of $D(w)$ attends some element of $E(w)$, so $m_{I, J}$ is not elusive.
    \end{proof}

    \begin{lemma}\label{lemma:elusive2}
        Let $m_{I, J}$ be as above. Then $(i_d, j_d)$ lies in $D(w)$.
    \end{lemma}
    \begin{proof}
        Let $m_{I, J}$ belong to some element $(i, j)$ of the essential set $E(w)$, so $i\geq i_d$ and $j\geq j_d$. Since $(i, j)$ lies in $E(w)$, we know that $w(i) > j$, and since $m_{I, J}$ is a $d\times d$ minor belonging to $(i, j)$ we also know that $r(i, j) = d-1$. Consider the sequence of elements $(i, j_a)$ for $1\leq a\leq d$. Since $w(i) > j$, each element $(i, j_a)$ lies in $D(w)$ unless $w(k) = j_a$ for some $k<i$. The definition of the rank function implies that there cannot be more than $d-1$ values of $k$ satisfying $w(k)\leq j_d$, so at least one element $(i, j_a)$ must lie in $D(w)$.
        
        Suppose $(i, j_c)$ lies in $D(w)$. If $c$ is strictly less than $d$, then by Lemma~\ref{lemma:elusive1} we know that $r(i, j_c)\geq c$. Thus we must have $w(k) \leq j_c$ for at least $c$ values of $k$ strictly less than $i$. This leaves at most $d-1-c$ values of $k$ satisfying $j_c < w(k) \leq j_d$. Thus at least one of the $(c-d)$ elements $(i, j_{c+1})$,\dots,$(i, j_d)$ lies in $D(w)$. Iterating this argument shows that $(i, j_d)$ lies in $D(w)$, and we see analogously that $(i_d, j)\in D(w)$. Thus $(i_d, j_d)$ lies in $D(w)$ as claimed.
    \end{proof}

    \begin{figure}[h]
        \begin{center}
        \begin{tikzpicture}[scale = 0.37]
            \draw[black, thick] (0, 0) -- (12, 0) -- (12, 12) -- (0, 12) -- (0, 0);
            \draw[black, thick] (4, 4) -- (6, 4) -- (6, 6) -- (4, 6) -- (4, 4);
            \draw[black, thick] (10, 0) -- (12, 0) -- (12, 2) -- (10, 2) -- (10, 0);
            \draw[black, thick] (4, 0) -- (6, 0) -- (6, 2) -- (4, 2) -- (4, 0);
            \draw[black, thick] (10, 4) -- (12, 4) -- (12, 6) -- (10, 6) -- (10, 4);
            \draw[black, dotted] (4, 0) -- (4, 12);
            \draw[black, dotted] (6, 0) -- (6, 12);
            \draw[black, dotted] (0, 4) -- (12, 4);
            \draw[black, dotted] (0, 6) -- (12, 6);
            \draw (11, 1) node{\scriptsize{$d-1$}};
            \draw (5, 5) node{\scriptsize{$k-1$}};
            \draw (5, 1) node{$\geq b$};
            \draw (11, 5) node{$\geq a$};
        \end{tikzpicture}
        \end{center}
        \caption{The situation of Lemma~\ref{lemma:reduction}. A value of $r$ in position $(i, j)$ means that there are $r$ points $(x, w(x))$ with $x\leq i$, $w(x)\leq j$.}
        \label{fig:rankconditions}
    \end{figure}
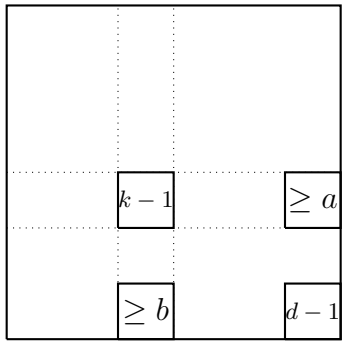
    
    \begin{lemma}\label{lemma:reduction}
        Let $m_{I, J}$ be as above and let $m_{I', J'}$ ($I'\subset I$, $J'\subset J$) be a $k\times k$ sub-minor ($k<d$) that is also a generator of $I_w$. Then the main antidiagonal of $m_{I', J'}$ is weakly northwest of the $(d-2)$th antidiagonal of $m_{I, J}$, where the first antidiagonal is the northwest corner.
    \end{lemma}
    \begin{proof}
        Let the southeast corner of $m_{I', J'}$ be $(i_a, j_b)$ for some $a$ and $b$ between $k$ and $d$. Without loss of generality we may assume that $I' = \{i_{a-k+1}, i_{a-k+2},\dots,i_a\}$ and $J' = \{j_{b-k+1}, j_{b-k+2},\dots,j_b\}$, since among all $k\times k$ sub-minors of $m_{I, J}$ with southeast corner $(i_a, j_b)$ this one has the most southeastern main antidiagonal. Writing down coordinates for points on the antidiagonals of $m_{I, J}$ and $m_{I', J'}$ reduces the proof to showing that $a+b\leq d+k-2$.

        Since $m_{I, J}$ has rank $d$ we know that $r(i_d, j_d) = d-1$. We claim that in addition $r(i_a, j_d)\geq a$. Let $c\in [a]$ be the greatest number such that $(i_c, j_d)$ lies in $D(w)$ (or take $c = 0$ if no such $c$ exists). Then by Lemma~\ref{lemma:elusive1} $r(i_c, j_d) \geq c$. Furthermore, since $(i_{c'}, j_d)\notin D(w)$ for the $(a-c)$ integers $c'$ satisfying $c<c'\leq a$ and $(i_d, j_d)\in D(w)$ by Lemma~\ref{lemma:elusive2}, we must have $w(i_{c'}) < j_d$ for all such $c'$. It follows that $r(i_a, j_d) \geq c+(a-c) = a$, since the rank function counts the number of $x\in [i_a]$ such that $w(x) < j_d$. The same argument shows that $r(i_d, j_b)\geq b$.

        We now know that $r(i_a, j_b) = k-1$, $r(i_d, j_d) = d-1$, $r(i_a, j_d) \geq a$, and $r(i_d, j_b)\geq b$. These rank conditions express the locations of points $(x, w(x))$ in the permutation matrix $M_w$. It follows that $r(i_a, j_d)+r(i_d, j_b)$ is bounded above by $r(i_a, j_b)+r(i_d, j_d)$ (see Figure \ref{fig:rankconditions}). Thus
        \[a+b\leq r(i_a, j_d) + r(i_d, j_b)\leq r(i_a, j_b) + r(i_d, j_b) = d+k-2.\qedhere\]
    \end{proof}
    
    \begin{proof}[Proof of Theorem~\ref{thm:main}]
        By Theorem~\ref{thm:schubertgrobner} and Theorem~\ref{thm:grobref2}(2), the elements of $\G'_w$ are the remainders obtained by dividing elusive minors $m_{I, J}$ in $\G_w$ by each other. Suppose that some term of $m_{I, J}$ is divisible by the lead term (\textit{i.e.}, the antidiagonal term) of another elusive minor $m_{I', J'}$. Then in particular that term is divisible by each variable in the main antidiagonal of $m_{I', J'}$. By Lemma~\ref{lemma:reduction} any variable $x_{ij}$ in the main antidiagonal of an elusive sub-minor $m_{I', J'}$ lies on or above the $(d-2)$th antidiagonal of $m_{I, J}$. We have therefore reduced the problem to showing that the determinant of a generic $d\times d$ matrix contains $2^{d-1}$ terms that avoid the variables $x_{i, j}$ with $i+j\leq d-1$. This is the case of Proposition~\ref{prop:extreme}.
    \end{proof}

\section{Vexillary $I_w$}\label{section4}

    In this paper we do not attempt to describe the reduced Gr\"obner basis $\G'_w$ explicitly. However, we can characterize the permutations $w$ such that the the minimal Gr\"obner basis $\G_w$ is already reduced. This happens precisely when $w$ is a \textit{vexillary} ($2143$-avoiding) permutation. Our proof uses the equivalent characterization of vexillary permutations as those $w$ such that all elements of the essential set $E(w)$ can be ordered into a list $\{e_1,e_2,\dots,e_n\}$ with each $e_i$ weakly southwest of $e_{i+1}$. This characterization appears as Remark 9.17 in \cite{Fulton}.
    
    \begin{proof}[Proof of Theorem~\ref{thm:vexillary}]
        Suppose first that $w$ is vexillary and let $g\in\G_w$ be an elusive minor belonging to some $e\in E(w)$. If $LT(g')$ divides some term of $g$, then $g'$ must belong to an $e'\in E(w)$ strictly northwest of the southeast corner of $g$. But then $e'$ is strictly northwest of $e$ since the southeast corner of $g$ is (weakly) northwest of $e$, so $w$ cannot be vexillary. It follows that no generators divide any terms of $g$, so this minimal Gr\"obner basis is in fact reduced.
        
        Conversely, suppose that $w$ is not vexillary, so there exist elements $e = (i, j)$ and $e' = (i', j')$ such that $e'$ is strictly southeast of $e$ (\textit{i.e.}, $i < i'$ and $j < j'$). After fixing $e$, we may choose $e'$ among the elements of $E(w)$ with this property to minimize the difference $k$ between $r' = r(i', j')$ and $r = r(i, j)$. Let $m_{I, J}$ be the maximally southeastern minor belonging to $e$ and let $m_{I', J'}$ be the maximally southeastern minor belonging to $e'$ and containing $m_{I, J}$ as a subminor. Then $m_{I, J}$ is an elusive minor belonging to $e$ by Claim 2.2 of \cite{GaoYong}, so to complete the proof it suffices to show that $m_{I', J'}$ is also elusive.
        
        If $m_{I', J'}$ is the maximally southeastern minor belonging to $e'$ then $m_{I',J'}$ is elusive by Claim 2.2 of \cite{GaoYong}. We therefore reduce to the case where
        $$I' = I\cup \{i'-k+1,i'-k+2,\dots,i'\}\textrm{ and } J' = J\cup \{j'-k+1,j'-k+2,\dots,j'\}.$$ Suppose that $m_{I', J'}$ attends $M^{[a, b]}$ for some $(a, b)\in E(w)$. Then $r(a, b)<r'$, so $r(a,b)-r<r'-r$, which contradicts our choice of $e'$ if $(a, b)$ is strictly southeast of $e$. But if $(a, b)$ is not strictly southeast of $e$, then by $m_{I', J'}$ attends $(a, b)$ if and only if $m_{I, J}$ does and we also obtain a contradiction.
    \end{proof}
    
\section{Binomial Schubert determinantal ideals}\label{section5}
    In order to establish Theorem~\ref{thm:binomial}, we first prove a pattern-avoidance characterization for permutations whose rank functions $r_w$ take values strictly less than $k$ on the essential set $E(w)$. We only need the special case $k=2$ of this proposition. However, the general proof requires no additional effort and may be of independent interest, so we provide it anyway.
    
    \begin{prop}\label{prop:patterns}
        The rank function $r_w$ satisfies $r_w(i, j) < k$ for all $(i, j)\in E(w)$ if and only if $w$ avoids the $k!$ patterns $\{v(k+2)(k+1)|v\in S_k\}$ in $S_{k+2}$. In particular, $r_w$ evaluates to $0$ or $1$ on each element of $E(w)$ if and only if $w$ avoids the patterns $1243$ and $2143$.
    \end{prop}
    \begin{proof}
        We prove the contrapositive. Suppose $w$ contains a pattern of the form $v(k+2)(k+1)$ for some $v\in S_k$ (written in one-line notation). Let $a_1 < a_2 < \dots < a_{k+2}$ witness this pattern, so $w(a_{k+2}) > w(a_i)$ for all $i\leq k$ and $w(a_{k+1})>w(a_{k+2})$. Let $b_i = w(a_i)$ for $1\leq i\leq k+2$. Then $d = (a_{k+1}, b_{k+2})$ lies in $D(w)$, as it is west of $(a_{k+1}, b_{k+1})$ and north of $(a_{k+2}, b_{k+2})$. Furthermore, since $d$ lies southeast of $(a_i, b_i)$ for each $i\leq k$, we know that $r_w(a_{k+1}, b_{k+2})\geq k$. Let $e = (i, j) \in E(w)$ be a southeast corner of the connected component of $D(w)$ containing $d$. Then $r_w(i,j) = r_w(a_{k+1}, b_{k+2})\geq k$.
        
        Conversely, suppose that for some $e = (i, j)$ in $E(w)$ we have $r_w(i, j)\geq k$. Since $r_w(i,j)\geq k$ there must exist points $(a_i, w(a_i))$ for $1\leq i\leq k$ such that $a_1<\dots<a_k<i$ and each $w(a_i)$ is less than $j$. Also, since $e$ lies in the Rothe diagram of $w$ it follows that $w(i) = b$ for some $b>j$ and $j = w(a)$ for some $a>i$. Thus we have points $(i, w(i))$ and $(a, w(a))$ such that $i<a$ and $w(i)>w(a)$. Putting these pieces together, we see that the sequence $(a_1,\dots,a_k, i, a)$ witnesses a pattern embedding of $v(k+2)(k+1)$ in $w$ for some $v\in S_k$.
    \end{proof}

    \begin{proof}[Proof of Theorem~\ref{thm:binomial}]
        By Proposition~\ref{prop:patterns}, the permutation $w$ avoids the patterns $1243$ and $2143$ if and only if $r_w(i, j)$ evaluates to $0$ or $1$ on every element of the essential set $E(w)$. This happens if and only if the Fulton generators of $I_w$ from Theorem~\ref{thm:schubref} are all binomials by definition. By Theorem~\ref{thm:schubertgrobner} it follows that the Fulton generators of $I_w$ are all binomials if and only if the elusive minors of the Gao--Yong minimal Gr\"obner basis $\G_w$ are all binomials. It therefore suffices to show that $I_w$ is a binomial ideal if and only if the set $\G_w$ of elusive minors consists of binomials. One direction is immediate, so we need only prove that if $I_w$ is a binomial ideal then $\G_w$ consists of binomials.

        We know from Proposition~\ref{prop:binomial} that $I_w$ is a binomial ideal if and only if its reduced Gr\"obner basis $\G'_w$ consists of binomials. We therefore need to show that if $\G_w$ contains a $k\times k$ minor for $k\geq 3$, then $\G'_w$ contains a generator with $\geq 3$ terms. This follows immediately from Theorem~\ref{thm:main}: $k\times k$ minors are polynomials of degree $k$, so if $\G_w$ contains such a minor then the corresponding generator in $\G'_w$ has at least $2^{3-1} = 4$ terms.
    \end{proof}

\section{Regularity of binomial $I_w$}\label{section6}
    As our final application, we will demonstrate that two recently-proven formulas for the  regularity of different classes of ideals coincide for binomial $I_w$. On one side, Rajchgot--Robichaux--Weigandt provided a formula in \cite[Theorem 1.5]{SchubReg} for vexillary Schubert determinantal ideals, which includes the binomial $I_w$ as a special case by Theorem \ref{thm:binomial}. On the other side, Almousa--Dochtermann--Smith gave a formula for the regularity of \textit{toric edge ideals} of bipartite graphs in \cite[Corollary 6.7]{ToricReg}. Binomial $I_w$ can be realized as toric edge ideals as shown by Portakal in \cite[pg. 7]{ToricEdgeIdeals}. We begin by recalling the definition of regularity. A standard reference is \cite[Chapter 20]{Eisenbud}.
    
    View $R = \Bbbk[x_{ij}]_{1\leq i, j\leq n}$ as a graded ring in the standard way (so each $x_{ij}$ has degree 1), and let $R(-a)$ be $R$ with all degrees shifted by $a$ (so each $x_{ij}$ has degree $1+a$). For a homogeneous ideal $I\subseteq R$, a \textit{(graded) free resolution} of $R/I$ is an exact sequence of free graded $R$-modules in the following form:
    \begin{equation*}
        0\to \bigoplus_{j\in\Z}R(-j)^{b_{ij}}\xrightarrow{\partial_i}\dots\xrightarrow{\partial_1}\bigoplus_{j\in\Z} R(-j)^{b_{0j}}\xrightarrow{\partial_0} R/I\to 0.
    \end{equation*}
    
    Exactness means that $\ker(\partial_i) = \textrm{im}(\partial_{i+1})$ for all $i$. The maps $\partial_i$ may be written as matrices, and a free resolution is called \textit{minimal} if no units of $R$ appear in these matrices. Equivalently, a minimal free resolution simultaneously minimizes all the numbers $b_{ij}$. It turns out that $R/I$ has a minimal free resolution, which is unique up to isomorphism \cite[Theorem 20.2]{Eisenbud}. The constants $b_{ij}$ appearing in the minimal free resolution are denoted by $\beta_{ij}$ and called the \textit{graded Betti numbers} of $R/I$. The \textit{(Castelnuovo--Mumford) regularity} of $R/I$ is defined as
    \begin{equation*}
        \reg(I) = \max\{j-i|\beta_{ij}\neq 0\}.
    \end{equation*}

    The following result is a well-known general fact about minimal free resolutions. Once one recalls the standard notions from homological algebra (which we omit here) it follows immediately from the fact that tensoring over a field is an exact functor.
        \begin{prop}\label{prop:homalg}
        Let $I \subseteq R = \Bbbk[x_1,\dots, x_r]$ be an ideal and suppose a generating set $G = \{g_1,\dots, g_s\}$ for $I$ can be partitioned into subsets $G_1$ and $G_2$ such that $G_1\subset R_1 = \Bbbk[x_1,\dots,x_k]$ and $G_2\subset R_2 = \Bbbk[x_{k+1},\dots,x_r]$ for some $k$. Let $I_1$ and $I_2$ be the ideals generated by $G_1$ in $R_1$ and $G_2$ in $R_2$ respectively. Let $F^1_\bullet$ and $F^2_\bullet$ be minimal free resolutions of $R_1/I_1$ and $R_2/I_2$. Then the minimal free resolution of $R/I$ is given by the tensor product $(F^1\otimes_\Bbbk F^2)_\bullet$. In particular, this implies that
        \begin{equation*}
            \beta_{a,b}(I) = \sum_{i+i' = a}\sum_{j+j' = b}\beta_{i, j}(I_1)\beta_{i', j'}(I_2)
        \end{equation*}
        and
        \begin{equation*}
            \reg(I) = \reg(I_1)+\reg(I_2).
        \end{equation*}
    \end{prop}

    \begin{definition}
        A permutation $v$ is \textit{dominant} if $r_v$ evaluates to $0$ on all of $E(v)$. When $v$ is dominant, its Rothe diagram consists of a single connected component $\lambda$, which we call its \textit{shape}. (This is a shape of a Young diagram.)
    \end{definition}

    \begin{definition}
        Let $w$ be a permutation and let $D_1,\dots, D_k$ be the connected components of $D(w)$. For $1\leq i\leq k$, let $v_i$ be a dominant permutation of shape $D_i$ and let $r_i$ be the value of the rank function on $D_i$. Let $u_i = 1^{r_i}\times v_i$ be the permutation such that $u_i(j)=j$ for $1\leq j\leq r_i$ and $u_i(j)=v_i(j)+r_i$ for $r_i+1\leq j$. The \textit{parts} of $I_w$ are the ideals $\{I_{u_i}\}_{i=1}^k$. The \textit{dominant part} of $I_w$ is $I_{u_i}$ for the unique $u_i$ that is dominant.
    \end{definition}

    In the case where $I = I_w$ is a binomial Schubert determinantal ideal, every generator of $I_w$ belonging to an element of $D_i$ can be viewed as a generator of the part $I_{u_i}$. Furthermore, no generators belonging to distinct $D_i$ share variables. This follows from the characterization of toric matrix Schubert varieties given by Escobar and M\'esz\'aros in \cite[Theorem 3.4]{ToricMatrixSchubs}. Combining this observation with Proposition~\ref{prop:homalg} and the fact that the dominant part of $I_w$ has regularity $0$ yields the following result:
    
    \begin{theorem}\label{thm:decomp}
        Let $I_w$ be a binomial Schubert determinantal ideal with dominant part $I_{u_0}$ and non-dominant parts $I_{u_1},\dots, I_{u_k}$. Then
        \begin{equation*}
            \reg(I_w) = \sum_{i=1}^k \reg(I_{u_i}).
        \end{equation*}
    \end{theorem}

    \begin{remark}
        This decomposition can be done for some non-binomial vexillary $I_w$. More precisely, it can be done whenever $w$ avoids the patterns $2143$, $14253$, $15243$, and their inverses, although we leave the proof for future work.
    \end{remark}

     Theorem~\ref{thm:decomp} implies that the Rajchgot--Robichaux--Weigandt and Almousa--Dochtermann--Smith formulas for the regularity of binomial $I_w$ are equal, provided they are equal for $I_{1\times v}$ when $v$ is a dominant permutation. We now present these formulas, starting with Rajchgot--Robichaux--Weigandt.

    \begin{definition}
        The \textit{canonical antidiagonal} of a partition $\lambda$ is the antidiagonal sequence $C_\lambda = \{e_1,...,e_k\}$ of maximum length in $\lambda$ such that $e_i = \{k-i+1, i\}$ for each $i$.
    \end{definition}

    \begin{theorem}[{Special Case of \cite[Theorem 1.5]{SchubReg}}]\label{thm:schubreg}
        Let $v$ be a dominant permutation of shape $\lambda$. Then $\reg(I_{1\times v}) = |C_\lambda|$.
    \end{theorem}

    Portakal expressed $I_{1\times v}$ ($v$ dominant) as the toric edge ideal of a bipartite graph:

    \begin{definition}
        The \textit{thickening} of a partition $\lambda$ is the partition $\bar{\lambda}$ with $\bar{\lambda}_1 = \lambda_1+1$ and $\bar{\lambda}_i = \lambda_{i-1}+1$ for all $i > 1$.
    \end{definition}

    \begin{definition}
        The \textit{graph} of a partition $\lambda$ with $m$ rows and $n$ columns is the connected bipartite graph $B_\lambda\subseteq K_{m, n}$ with edges $(i, j)$ whenever $(i, j)$ lies in the diagram of $\lambda$.
    \end{definition}

    \begin{theorem}[{\cite[pg. 7]{ToricEdgeIdeals}}]\label{thm:schubgraph}
        Let $v$ be dominant of shape $\lambda$. Then $I_{1\times v}$ is the toric edge ideal of the graph $B_{\bar\lambda}$ of the thickening $\bar{\lambda}$.
    \end{theorem}

    Using Theorem~\ref{thm:schubgraph}, the following formula of Almousa--Dochtermann--Smith also gives the regularity of $I_{1\times v}$ for $v$ dominant.

    \begin{definition}
        Let $B\subset K_{m, n}$ be a bipartite graph and $S$ a subgraph. The \textit{recession graph} $R(S;B)$ is the directed bipartite graph built from $B$ by directing the edges in $G\setminus S$ from $[m]$ to $[n]$ and making the edges of $S$ bidirectional.
    \end{definition}

    \begin{definition}
        The \textit{recession connectivity} $r(B)$ of a bipartite graph $B$ is the maximum number of components in a subgraph $S\subset B$ such that $R(S;B)$ is strongly connected.
    \end{definition}

    \begin{theorem}[{\cite[Corollary 6.7]{ToricReg}}]\label{thm:toricreg}
        If $B$ is a connected bipartite graph, then the regularity of the toric edge ideal $I_B$ is $r(B)-1$.
    \end{theorem}

    Since the formulas in Theorem~\ref{thm:schubreg} and Theorem~\ref{thm:toricreg} both compute $\reg(I_{1\times v})$, they must be equal. We conclude this paper with a direct proof that the formulas agree.

    \begin{lemma}\label{lemma:upperbound}
        For any partition $\lambda$, $r(B_{\bar{\lambda}}) \leq |C_\lambda|+1$.
    \end{lemma}
    \begin{proof}
        Note that $r(B_{\bar{\lambda}})$ is bounded above by $m(B_{\bar{\lambda}})$, since taking one edge from each connected component of a subgraph $S\subset B_{\bar\lambda}$ yields a matching. If the number of rows or columns in $\lambda$ is equal to $|C_\lambda|$, then $|C_{\bar\lambda}| = |C_\lambda|+1$. Since $|C_\mu| = m(B_\mu)$ for any partition $\mu$, $r(B_{\bar\lambda})\leq |C_\lambda|+1$ in this case. Now suppose $\lambda$ has more than $|C_\lambda|$ rows and columns, so $|C_{\bar\lambda}| = |C_\lambda| + 2$. This means there exist $(a+1, b)$ and $(a, b+1)$ in $C_{\bar{\lambda}}$ such that $(a+1, b+1)$ is not in ${\bar\lambda}$. 

        Now let $S\subset B_{\bar\lambda}$ be such that $R(S; B_{\bar\lambda})$ is strongly connected. Note that $S$ must contain an edge $e\in[a]\times[b]$, since otherwise $R(S; B_{\bar\lambda})$ cannot have a path from the column vertex $1$ to the edge vertex $1$. The number of connected components of $S$ is then witnessed by a matching in $B_{\bar\lambda}$ containing $e$. Removing the row and column of ${\bar\lambda}$ intersecting in $e$ yields an auxiliary partition $\nu$. It is clear that $|C_\nu| \geq |C_{\bar\lambda}| - 2$, and the reverse inequality must hold because $C_{\bar{\lambda}}$ passes through $(a+1, b)$ and $\nu$ contains neither $(a+1, b)$ nor $(a, b)$. This implies that the number of components of $S$ is bounded above by $m(B_\nu)+1 = |C_{\bar\lambda}| -2+1 = |C_\lambda| + 1$, completing the proof.
    \end{proof}

    \begin{theorem}\label{thm:formulasagree}
        For any partition $\lambda$, $r(B_{\bar\lambda}) = |C_\lambda|+1$. 
    \end{theorem}
    \begin{proof}
        We explicitly construct a subgraph $S\subset B_{\bar\lambda}$ with $|C_\lambda| +1$ connected components such that $R(S;B)$ is strongly connected. Let $S' = C_\lambda\cup\{(1, j),(i, 1) : i, j>|C_\lambda|\}$. Now let $S = \{(i+1, j+1)|(i, j)\in S'\}\cup \{(1, 1)\}$, so $S$ consists of $S'$ (regarded as a subset of $\bar\lambda$) along with $(1, 1)$. It is clear that $S$ has $|C_\lambda|+1$ components when viewed as a collection of edges in the graph $B_{\bar\lambda}$. It is also straightforward to verify that $R(S;B_{\bar\lambda})$ is strongly connected. This proves that $r(B_{\bar\lambda})\geq |C_\lambda|+1$. Equality follows by Lemma~\ref{lemma:upperbound}.
    \end{proof}

    \begin{example}
        With $\lambda = (6, 4, 1, 1, 1)$ and $\bar{\lambda} = (7, 7, 5, 2, 2, 2)$, the subgraph $S\subset B_{\bar\lambda}$ constructed in Theorem~\ref{thm:formulasagree} corresponds to the $\bullet$'s in $\bar\lambda$. In this case $|C_\lambda| = 3$ while $\lambda$ has 5 rows and 6 columns, which implies that $|C_{\bar\lambda}| = |C_\lambda| + 2$. In the notation of the proof of Lemma~\ref{lemma:upperbound} we have $[a]\times[b] = [3]\times[2]$. This box is highlighted in yellow on the diagram.

        
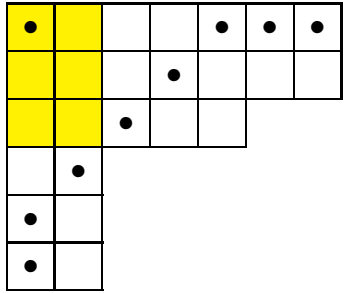
\begin{center}
        \begin{ytableaushort}
            {\bullet\none\none\none\bullet\bullet\bullet, \none\none\none\bullet, \none\none\bullet, \none\bullet, \bullet, \bullet}
            *{7, 7, 5, 2, 2, 2}
            *[*(yellow)]{2, 2, 2}
        \end{ytableaushort}
        \end{center}
    \end{example}

\section*{Acknowledgements}
We thank Laura Escobar for raising the question that led to Theorem~\ref{thm:binomial} during her visit to UIUC supported by an NSF RTG in Combinatorics (DMS 1937241).
We also thank Nathan Hayes, Tyler Lawson, Gidon Orelowitz  and Alexander Yong for helpful conversations about this material. We were partially supported by a Susan C.~Morosato IGL graduate student scholarship and an NSF RTG in Combinatorics (DMS 1937241). This material is based upon work supported by the National Science Foundation Graduate Research Fellowship Program under Grant No. DGE 21-46756.



\begin{thebibliography}{99}

\bibitem{ToricReg}
Almousa, Ayah; Dochtermann, Anton; Smith, Ben. Root polytopes, tropical types, and toric edge ideals, preprint, 2022. \textsf{arXiv: 2209.09851} 

\bibitem{CLO}
Cox, David A.;Little, John; O'Shea, Donal. Ideals Varieties, and Algorithms: An Introduction to Computational Algebraic Geometry and Commutative Algebra. Undergraduate Texts in Mathematics. Springer-Verlag, Berlin, Heidelberg, 2007. 3rd ed.

\bibitem{Eisenbud}
Eisenbud, David. Commutative Algebra: With a View Toward Algebraic Geometry. Spring, New York, NY, 1995. \textsf{ISBN: 978-1-4612-5350-1}

\bibitem{EisenbudSturmfels}
Eisenbud, David; Sturmfels, Bernd. Binomial Ideals. Duke Mathematical Journal, 84(1):1--45, 1996. \textsf{DOI: 10.1215/S0012-7094-96-08401-X}

\bibitem{ToricMatrixSchubs}
Escobar, Laura; Me\'esz\'aros, Karola. Toric matrix Schubert varieties and their polytopes. Proceedings of the American Mathematical Society, 144(12):5081--5096, 2016. \textsf{ISSN: 00029939, 10886826}

\bibitem{Fulton}
Fulton, William. Flags, Schubert polynomials, degeneracy loci, and determinantal formulas. Duke mathematical Journal, 65(3):381--420, 1992. \textsf{DOI: 10.1215/S0012-7094-92-06516-1}

\bibitem{GaoYong}
Gao, Shiliang; Yong, Alexander. Minimal equations for matrix Schubert varieties, preprint, 2022. \textsf{arXiv: 2201.06522}

\bibitem{KnutsonMiller}
Knutson, Allen; Miller, Ezra. Gr\"obner geometry of Schubert polynomials. Ann. of Math, 161(3):1245--1318, 2005. \textsf{DOI: 10.4007/annals.2005.161.1245}

\bibitem{KMY}
Knutson, Allen; Miller, Ezra; Yong, Alexander. Gr\"obner geometry of vertex decompositions and of flagged tableaux. Crelle, 2009(630):1--31, 2009. \textsf{DOI: 10.1515/CRELLE.2009.033}

\bibitem{Kremer}
Kremer, Darla. Permutations with forbidden subsequences and a generalized Schr\"oder number. Discrete Mathematics, 218(1):121--130, 2000. \textsf{DOI: 10.1016/S0012-365X(99)00302-7}

\bibitem{MacDonaldNotes}
MacDonald, Ian. Notes on Schubert Polynomials. In Publications of LACIM, Vol. 6, 1991.

\bibitem{SchubReg2}
Pechenik, Oliver; Speyer, David; Weigandt, Anna. Castelnuovo--Mumford regularity of matrix Schubert varieties, preprint, 2021. \textsf{arXiv: 2111.10681}

\bibitem{ToricEdgeIdeals}
Portakal, Irem. Rigid toric matrix Schubert varieties, preprint, 2022. \textsf{arXiv: 2001.11949}

\bibitem{SchubReg}
Rajchgot, Jenna; Robichaux, Colleen; Weigandt, Anna. Castelnuovo--Mumford regularity of ladder determinantal varieties and patches of Grassmannian Schubert varieties. Journal of Algebra, 617:160--191, 2023. \textsf{DOI: 10.1016/j.jalgebra.2022.11.001}

\end{thebibliography}
\end{document}